\documentclass[11pt,a4paper]{amsart}

\usepackage{amsmath}
\usepackage{dsfont}
\usepackage{color}
\usepackage{cancel}
\usepackage{amssymb}
\usepackage[mathscr]{eucal}
\usepackage{parskip}
\usepackage{epstopdf}
\usepackage[marginpar=1.5cm]{geometry}
\usepackage{amsaddr}

\usepackage{graphicx}
\usepackage{float}

\numberwithin{equation}{section}

\newtheorem{theorem}{Theorem}[section]
\newtheorem{definition}[theorem]{Definition}
\newtheorem{lemma}[theorem]{Lemma}
\newtheorem{remark}[theorem]{Remark}

\newtheorem{proposition}[theorem]{Proposition}

\newtheorem{example}[theorem]{Example}

\newcommand{\N}{\mathbb{N}}
\newcommand{\R}{\mathbb{R}}

\newcommand{\mP}{\mathbb{P}}

\newcommand{\T}{\mathbb{T}}
\newcommand{\Z}{\mathbb{Z}}

\newcommand{\fiber}{(\tau, \omega)}

\newcommand{\cH}{\mathcal{H}}

\newcommand{\cF}{\mathcal{F}}

\newcommand{\cB}{\mathcal{B}}

\newcommand{\opintb}[1]{\big(#1\big)}

\newcommand{\set}[1]{\{#1\}}
\newcommand{\setb}[1]{\big\{#1\big\}}
\newcommand{\setB}[1]{\Big\{#1\Big\}}

\newcommand{\vp} {{\varphi}}
\newcommand{\fa} {\text{for all }\,}
\newcommand{\fad} {\quad \text{for all }\,}

\newcommand{\faa} {\text{for almost all }\,}

\newcommand{\dist}{\operatorname{dist}}

\newcommand{\rmd}{\mathrm{d}}

\newcommand{\mand}{\text{ and }}

\newcommand {\na}{nonautonomous }
\newcommand{\rds} {random dynamical system}
\newcommand{\nrds} {nonautonomous random dynamical system}

\begin{document}

\title[A random dynamical systems perspective on stochastic resonance]{A random dynamical systems perspective \\ on stochastic resonance}

\author{Anna Maria Cherubini}
\address{\vspace{-0.3cm}Dipartimento di Matematica e Fisica ``Ennio De Giorgi'', Universit\`a del Salento\\ I-73100 Lecce, Italy}
\author{\vspace{-0.6cm}Jeroen S.W.~Lamb, Martin Rasmussen}
\address{\vspace{-0.3cm}Department of Mathematics, Imperial College London\\ 180 Queen's Gate, London SW7 2AZ, United Kingdom}
\author{\vspace{-0.6cm}Yuzuru Sato}
\address{\vspace{-0.3cm}Department of Mathematics, Hokkaido University\\ Kita 12 Nishi 6, Kita-ku, Sapporo, Hokkaido 060-0812, Japan\\
 London Mathematical Laboratory\\ 14 Buckingham Street, London, WC2N 6DF,United Kingdom}

\date{\today}


\thanks{The authors gratefully acknowledge financial support from the following sources: EPSRC Career Acceleration Fellowship EP/I004165/1 (MR), EPSRC Mathematics Platform Grant (YS), London Mathematical Laboratory external fellowship (YS), JSPS Grant-in-Aid for Scientific Research  No.~24540390 (YS), Russian Science Foundation grant 14-41-00044 at the Lobachevsky University of Nizhny Novgorod (JL). This research has also benefitted from funding from the European Union's Horizon 2020 research and innovation programme for the ITN CRITICS under Grant Agreement number 643073.}

\maketitle

\begin{abstract}
We study stochastic resonance in an over-damped approximation of the stochastic Duffing oscillator from a random dynamical systems point of view.  We analyse this problem in the general framework of \rds s with a nonautonomous forcing. We prove  the existence of a unique global attracting random periodic orbit  and  a stationary  periodic measure.  We use the stationary periodic  measure to define an indicator for the stochastic resonance.

\smallskip
\noindent
\textbf{Keywords.} Markov measures, nonautonomous random dynamical systems, random attractors, stochastic resonance.

\smallskip
\noindent
\textbf{Mathematical Subject Classification.}  37H10, 37H99, 60H10.

\end{abstract}

\section{Introduction}
Stochastic resonance is the remarkable physical phenomenon where a signal that is normally too weak to be detected by a sensor, can be boosted by adding noise to the system. It  has been initially proposed in the context of climate studies,  as an explanation of the recurrence of ice ages \cite{Benzi_81_1,Benzi_82_1,Benzi_83_1,Nicolis_81, Nicolis_82}, and subsequently the  phenomenon has been reported in other fields, such as biology and neurosciences,  and extensively studied in many different physical settings. It is not possible here to  account  for the huge literature on the subject but we  refer to \cite {Gammaitoni_98_1} for a  comprehensive review, while an exhaustive discussion of the literature in different fields can be found in \cite {McDonnell_08_1}. Of particular relevance  are the mathematical studies of the phenomenon  in \cite  {Berglund_06_1} and \cite {Hermann_14_1}.


In this paper, we study one of the models  for stochastic resonance 
from a random dynamical systems point of view. Despite the obvious merit of gaining insight in stochastic processes from a dynamical systems standpoint,  and various research programmes in this direction (see e.g.~\cite{Arnold_98_1}), the mathematical field of random dynamical systems is still in its infancy. Our study of stochastic resonance, as a prototypical dynamical phenomenon in stochastic systems, illustrates how this approach
provides additional insights to phenomena of broad physical interest. In the process, we extend the existing random dynamical systems theory in the direction of nonautonomous stochastic differential equations, to aid the analysis of the particular model at hand. We note in this context that whereas autonomous stochastic differential equations are widely studied, nonautonomous stochastic differential equations received much less attention \cite{Caraballo_03_3,Carvalho_13_1,Crauel_11_1,Cong_02_1} and we also mention  \cite {Wang_14_1,Zhao_09,Zhao_15} for  pioneering work  on random periodic solutions of random dynamical systems.

We study one of the simplest stochastic differential equations used to model stochastic resonance, commonly motivated by taking an overdamped limit of a stochastically driven Duffing oscillator \cite{Gammaitoni_98_1}:
\begin{equation}\label{sr}
 \rmd x= \opintb{\alpha x- \beta x^3}\rmd t+ A \cos \nu t\  \rmd t + \sigma \rmd W_t, \quad \alpha, \beta , \sigma>0, \quad
 x \in \R,
 \end{equation}
where $(W_t)_{t\in\R}$  denotes a Wiener process.
The full model describes
a damped particle in a periodically oscillating double-well potential in the presence of noise. The  periodic  driving tilts the double-well   potential  asymmetrically  up  and down,   raising  and  lowering   the  potential barrier. If the periodic  forcing  alone  is too  weak  for  the  particle  to leave one   potential  well,  the noise strength can be tuned so that   hopping   between   the  wells     is  synchronised with the  periodic  forcing and  the average  waiting  time between  two  noise-induced hops is  comparable with the period of  the  forcing.  For increasing noise strength, the  periodicity is lost and the hopping becomes increasingly random.
It is important to observe that (\ref{sr}) has, in addition to the noise, also an explicit deterministic dependence on time. We refer to such systems as
{\it nonautonomous stochastic differential equations}. In the model at hand, the deterministic time-dependence is periodic, which facilitates the analysis in a crucial way.

We establish a random dynamical systems point of view for nonautonomous stochastic differential equations. In this context we aim to describe the long-time asymptotic behaviour of (\ref{sr}) in terms of (random) attractors and we prove the following: 
\begin{theorem}\label{th1}
The SDE (\ref{sr}) has a unique globally attracting random periodic orbit.
\end{theorem}
In terms of the dynamics, if we   denote by $\Phi(t,\tau,\omega)$ the \rds \  induced by (\ref{sr}) on $\R$ from time $\tau$ to $\tau+t$  and  for  a realisation $\omega$ of the Wiener process,
this means that for any bounded set $C\subset \R$, the limit
$\lim_{t\to\infty} \Phi(t,\tau -t,\theta_{-t}\omega)C$ is a single point $A(\tau,\omega)$ for all $ \tau$ and almost all $\omega$.
We note that $A(\tau,\cdot)$ is a random variable that evolves under the stochastic flow as $\Phi(t,\tau,\omega)A(\tau,\omega)=A(\tau+t,\theta_t(\omega))$, with
$A(\tau+T, \omega)=A(\tau,\omega)$ where $T=\frac{2\pi}{\nu}$, justifying the nomenclature \emph{random periodic orbit}.

The attractor provides all the dynamical information for the system and it is accompanied by a natural set of probability measures: first of all a singleton distribution $\delta_{A({\tau,\omega})}$ associated to the random periodic orbit.
It is natural now to consider the measure $\rho_t(B):=\int \delta_{A({t,\omega})}(B){\rm d}\omega$ on  measurable sets $B\subset \R$, for $t\in \R$. By ergodicity, this measure  provides a probabilistic description of orbits of the random fixed
point  starting at time $\tau$ under the time-$T$ map, in the sense that the expected frequency to visit a subset $B\subset \R$ is equal to the $\rho_\tau$-measure of this subset:
\[
\rho_\tau(B)=\lim_{N\to\infty} \frac{1}{N}\sum_{n=1}^N \mathds{1}_{B}(\Phi(nT,\tau,\omega)A(\tau,\omega))
\]
for almost all $\omega$. An illustration of the  density of the random periodic orbit is given in Fig.~\ref{periodicdistrib}: importantly, it depends on $\tau$, and it is $T$-periodic, i.e. $\rho_\tau= \rho_{\tau+T}$, a s a consequence of  \ref{th1}.
 \begin{figure}[H]  \label {periodicdistrib}
  \includegraphics[scale=0.7]{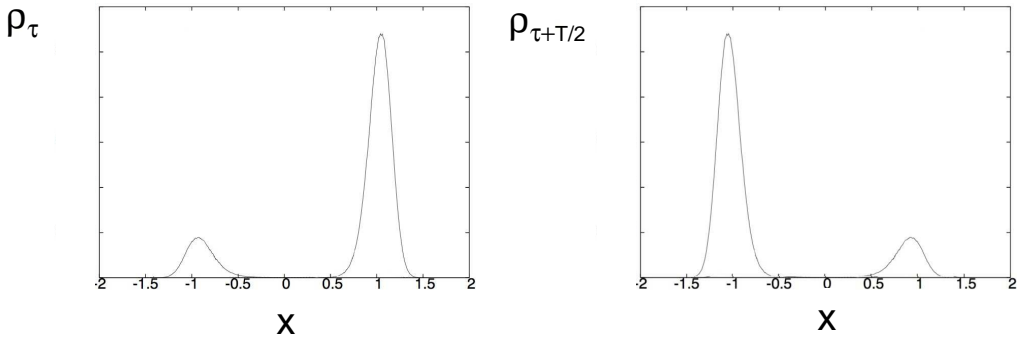}
  \caption{The  Lebesgue density of the measure $\rho_t$ 
   for $t=\tau$ and   $t= \tau+{T\over 2}.$ The values of the parameters are:
$\alpha=\beta=1, \ A=0.12, \ \nu=0.001$, $ \sigma=0.285$ and $\tau=0$.}
\end{figure}

We note that this result depends heavily on observing the orbit at the time step $T$. If a time step $T'$ is incommensurate with the period $T$, then
it is easy to see that
\[\overline{\rho}(B)=
\lim_{N\to\infty} \frac{1}{N}\sum_{n=1}^N \mathds{1}_{B}(\Phi(nT',\tau,\omega)A(\tau,\omega))=
\lim_{t\to\infty} \frac{1}{t}\int_{0}^t \mathds{1}_{B}(\Phi(s,\tau,\omega)A(\tau,\omega))\rmd s.
\]
with $\overline{\rho}:=\frac{1}{T}\int_{0}^T \rho_t \rmd t$. For fundamental research on ergodic theory and probability measures of periodic  random dynamical systems,  see  \cite{Zhao_15}.

\begin{figure}[H]
  \includegraphics[scale=0.65]{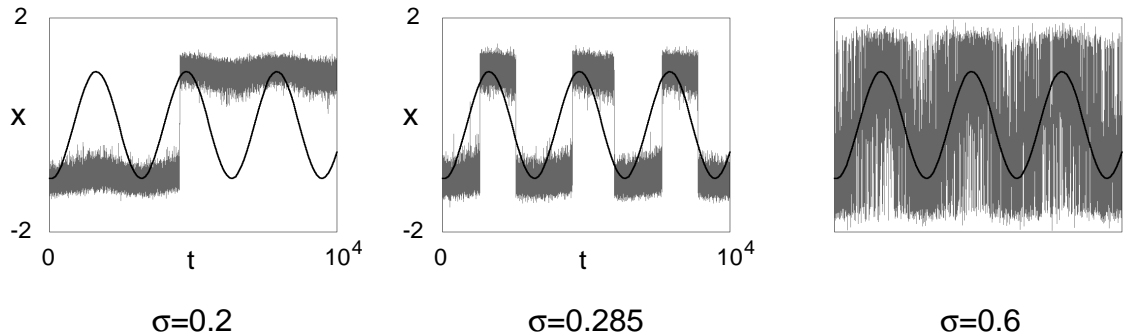}
   \caption{Time series of orbits for \eqref{sr} at increasing values of noise ($\sigma=0.2, 0.285, 0.6$). The case $\sigma=0.285$  corresponds to stochastic resonance. The graph of the nonautonomous driving term $A\cos(\nu t)$ is also depicted for reference.} 
   \label {figsr}
\end{figure}

We now proceed to discuss the phenomenon of stochastic resonance within the above point of view. Stochastic resonance in the context of our model \eqref{sr} is measured in terms of the enhancement of the periodic behaviour of the nonautonomous forcing as a function of the size of the noise. In Fig.~\ref {figsr} we present representative time series of \eqref{sr}  for three different noise amplitude levels  $\sigma$: before, during and after the stochastic resonance regime. The phenomenon of stochastic resonance is characterized by a $T$-period hopping between the left and right potential wells, present in the middle plot but absent in the leftmost and rightmost plots.

The resonant regime can be identified in various ways. The classical experimental indicator is the signal to noise ratio, see e.g.~\cite{Gammaitoni_98_1}.
Our establishment of the existence of a unique globally attracting random periodic point with invariant measures $\rho_t$  provides the opportunity to define other indicators with more mathematically rigorous footing. As proposed already by  \cite {Gammaitoni_98_1} (but without a rigorous discussion of existence), one can for instance consider the expectation
$\bar x(t) =\max_{0\le t\le T}  \int_\R x \rmd \rho_t(x)$, which due to the periodicity of $\rho_t$ is also $T$-periodic. The size of the amplitude of this oscillating function,
$\bar x=\max_{0\le t< T} | \bar{x}(t)|$ is a natural indicator for stochastic resonance.

However, as $\bar{x}$ does not really measure the likelihood of a time series to hop from left to right in resonance with the driving frequency, we here propose an alternative indicator which directly relates to the amount of transport between the wells across the barrier at $x=0$ over a time period $T$. Define the two probabilities
\begin{displaymath}
p^-:=\frac{\max_{0\le t < T} \rho_t ((-\infty,0] ) -\min_{0\le t < T} \rho_t ((-\infty,0] )}{\max_{0\le t < T} \rho_t ((-\infty,0] ) }
\end{displaymath}
and 
\begin{displaymath}
p^+:=\frac{\max_{0\le t < T} \rho_t ([0,\infty) ) -\min_{0\le t < T} \rho_t ([0,\infty) )}{\max_{0\le t < T} \rho_t ([0,\infty) ) }\,,
\end{displaymath}
and note that $p^-$ is a lower bound for the probability for a particle to move from the left to the right well, while $p^+$ is a lower bound for the probability for a particle to move from the right to left. In general, these two probabilities do not coincide, although they do for the stochastic differential equation \eqref{sr}. The product of these two probabilities
\begin{displaymath}
  p:= p^-p^+
\end{displaymath}
is a lower bound for a particle to switch the well two times within the period $T$, and thus serves as an indicator for stochastic resonance.
In Fig.~\ref{response} we present a comparison between the indicators $p$ and $\bar{x}$ for different values of the noise strength $\sigma$, showing that both maximize
at the same noise strength in this specific example.
\begin{figure} [H]
  \includegraphics[scale=0.6]{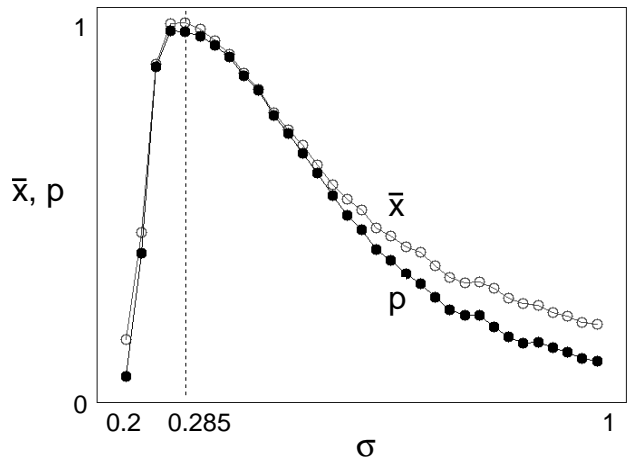}
  \caption{The indicators $p$ (filled dot) and $\bar x$ (empty dot) as a function of $\sigma$. Both indicators are maximised in the resonant regime.}
  \label{response}
\end{figure}

The paper is organised as follows. In Section 2, we define the framework for nonautonomous random dynamical systems.
In  Section 3, we prove the existence of global nonautonomous random attractors for a class of \na stochastic differential equations including \eqref {sr};  more general results on the existence of attractors for \nrds s are developed in the Appendix. In Section 4, we define periodic measures for  nonautonomous random dynamical systems and prove that  for the class of systems defined in Section 3 there exists a unique attracting random periodic orbit.

\textbf{Acknowledgments.} The authors would like to thank Sergey Zelik (University of Surrey) and Armen Shirikyan (Universit\'e de Cergy-Pontoise) for useful discussions. 

\section{Nonautonomous \rds s}

In this section, we define the fundamental objects we need to study stochastic resonance in the framework of the theory of random dynamical systems.  Similarly to the autonomous case \cite {Arnold_98_1}, the noise of a nonautonomous random dynamical system is modelled by a base flow $\theta.$ Let $(\Omega,\cF,\mP)$ be a probability space with a $\sigma$-algebra $\cF$ and a probability measure $\mP$, and let $\T$ be a time set (given by either $\R$ or $\Z$). A $(\cB(\T)\otimes \cF, \cF)$-measurable function $\theta:\T\times \Omega \to \Omega$  is called \emph{measurable dynamical system} if $\theta(0,\omega) =\omega$ and $\theta(t+s,\omega) =  \theta(t, \theta(s, \omega))$ for all $t,s\in\T$ and $\omega\in\Omega.$ We will assume that $\theta$ is \emph{measure preserving} or \emph{metric}, i.e.~$\mP\theta(t,A) = \mP A$ for all $t\in\T$ and $A\in\cF$, and we will call $\theta$ \emph{ergodic}  if  the invariant sets for the flow have trivial measure. We will use the abbreviation $\theta_t\omega$ for $\theta(t,\omega).$

In contrast to the autonomous case, the dynamics of nonautonomous \rds s depends also on the initial time, rather than only on $\omega\in\Omega$ and the elapsed time.

\begin{definition} [Nonautonomous and periodic random dynamical system]  \label{nRDS}
  Let $X$ be a Polish space with metric $d.$ A \emph{nonautonomous random dynamical system} on $X$ is a pair $(\theta, \Phi)$, where $\theta:\T\times \Omega \to \Omega$ is a metric dynamical system defined on the probability space $(\Omega,\cF,\mP)$, and the so-called \emph{cocycle} $\Phi: \T \times \T \times \Omega \times X \to X$ is a $(\cB(\T) \otimes \cB(\T) \otimes \cF \otimes \cB(X),\cB(X))$-measurable mapping with the following properties:
  \begin{itemize}
    \item[(i)] $\Phi(0,\tau,\omega,x)= x$ for all $\tau \in \T$, $x\in X$ and  $\faa \omega \in \Omega$,
    \item[(ii)] $\Phi(t+s,\tau, \omega,x) = \Phi(t,\tau + s, \theta_s\omega, \Phi(s,\tau, \omega,x))$ for all $t,s, \tau \in \T$, $x\in X$ and for almost all $\omega\in\Omega.$
  \end{itemize}
  We assume that $\Phi(\cdot,\cdot, \omega,\cdot): \T \times \T \times X \mapsto \T \times \T \times X $ is continuous for almost all $\omega \in \Omega$, and we will use the notation $\Phi(t,\tau, \omega)x$ for $\Phi(t,\tau, \omega,x).$ The nonautonomous random dynamical system $(\theta,\Phi)$  is called \emph{periodic random dynamical system} if there exists a $T>0$ such that
  \begin{displaymath}
    \Phi(t,\tau + T , \omega,x) = \Phi(t,\tau, \omega,x) \fad t,\tau \in \T, x\in X \mand \faa \omega\in\Omega\,.
  \end{displaymath}
\end{definition}

We are mainly interested in continuous-time nonautonomous random dynamical systems, which are generated by a stochastic differential equation (SDE for short) of the form
\begin{equation}\label{sde}
  \rmd x=f(t,x) \rmd t+ \sigma \rmd W_t
\end{equation}
where $ t,x \in \R \mand  \sigma>0$, $(W_t)_{t\in\R}$ is a Wiener process and $f:\R^2\to\R$ is continuously differentiable. For conditions on existence and uniqueness of  global solutions of \eqref {sde}, see  \cite {Arnold_74_1, Prevot_03_1}. In this case, the underlying  model for the noise is given by the Wiener space  $\Omega:= C_0(\R,\R):=\set{\omega\in C(\R,\R): \omega(0)=0}$ equipped with the compact-open topology and the Borel $\sigma$-algebra $\cF:=\cB(C_0(\R, \R)).$ $\mP$ is the Wiener probability measure on $(\Omega, \cF)$ and the evolution of noise is described by the Wiener shift $\theta: \R\times \Omega\to\Omega$, defined by $\theta(t, \omega(\cdot)):= \omega(\cdot+t)-\omega(t).$ The shift is ergodic \cite{Arnold_98_1}. The cocycle is defined by
\begin{equation}\label{Phi}
  \Phi ( t, \tau,\theta_{\tau}\omega, x):=  \mathcal X(t+\tau,\tau,\omega,x)\,,
\end{equation}
where $\mathcal X(t,\tau,\omega,x)$ is the stochastic flow of \eqref{sde}, i.e.~a pathwise solution for an  initial time $\tau\in\R$ and initial condition $x\in \R.$

The nonautonomous random dynamical system $\Phi $ is periodic if the function $f$ is periodic in $t.$ Note that $\Phi $ is \emph{order preserving}, i.e.~for $x,y \in \R$ with $x\le y$, we have
$\Phi(t,\tau, \omega,x) \le \Phi(t,\tau, \omega,y)$ for all $t,\tau \in \T$ and almost all $\omega \in \Omega.$

The following example describes how four homeomorphisms generate a periodic random dynamical system in discrete time.

\begin{example}[Discrete-time periodic random dynamical system]\label{discrete}
  Consider a metric space $X$ and four homeomorphisms $h^i_j:X\to X$, where $i,j\in\set{0,1}.$ We want to study the random dynamics if $h^i_j$ is used with probability $p_j \in[0,1]$ at either even times ($i=0$) or odd times ($i=1$). We assume that $p_0+p_1 = 1.$ We define $\Omega:=\setb{\omega = (\dots, \omega_{-1},\omega_0,\omega_1, \dots):\omega_i\in\set{0,1}}.$ For fixed $x_1,\dots,x_n\in \set{0,1}$, the set
  $$I_{x_0,\dots,x_{n}}:=\setb{\omega\in \Omega:\omega_i = x_i \text{ for all }  i\in\set{0,\dots,n}  }$$
  is called a \emph{cylinder set}. The set of cylinder sets forms a semi-ring, and we define $\cF$ to be the $\sigma$-algebra generated by this semi-ring.
  We define $\mP$ on cylinder sets by $\mP (I_{x_0,\dots,x_n}):=\prod_{i=0}^n p_{x_i}$, and then we extend $\mP$ to $\cF.$ The dynamics on $\Omega$ is given by the left shift $(\theta(\omega))_i=\omega_{i+1}$, and we
  define the cocycle by $\vp(1,m,\omega,x):=(h_{\omega_0}^{m\hspace{-0.1cm}\mod 2})(x)$ for $m\in\Z.$ The nonautonomous random dynamical system ($\theta,\vp$) is two-periodic.
\end{example}



Let $(\Omega,\cF,\mP)$ be a probability space, and let $\T$ be a time set and $X$ be a Polish space. A $(\cB(\T) \otimes \cF \otimes \cB(X))$-measurable set $M\subset \T \times \Omega \times X$ is called a \emph{nonautonomous random set}, and the set $M(\tau,\omega):=\{x\in X: (\tau,\omega,x)\in M \}$ is called the $(\tau,\omega)$-\emph{fiber} of $M.$ If every fiber of $M$ is closed (compact, or bounded, respectively), then $M$ is called \emph{closed} (\emph{compact}, or \emph{bounded}, respectively). If all fibers are singletons, we will call $M$ a \emph{nonautonomous random point}.

A nonautonomous random set $M$ is called \emph{invariant} with respect to a nonautonomous random dynamical system $(\theta,\Phi)$ if
\begin{displaymath}
  \Phi(t,\tau, \omega) M(\tau,\omega) = M(\tau + t,\theta_t\omega) \fad t, \tau, \in \T \text{ and almost all }\omega\in\Omega\,.
\end{displaymath}
Invariant nonautonomous random sets can be constructed easily. For instance, given $x\in X$, the set defined by $M(\tau, \omega):=\Phi(\tau,0,\theta_{-\tau}\omega,x)$ for $\tau\in \T$ and $\omega\in\Omega$ is invariant.

While the construction of invariant nonautonomous random sets is straightforward, so-called invariant periodic random sets, as defined below, are nontrivial objects.

\begin{definition} [Periodic random sets and random periodic orbits]
  An invariant nonautonomous random set $M$ is called an \emph{invariant periodic random set} if  there exists a $T>0$ such that
  \begin{displaymath}
    \Phi(T,\tau, \omega) M(\tau,\omega) = M(\tau,\theta_T\omega)  \fad \tau \in \T \text{ and almost all } \omega\in\Omega\,.
  \end{displaymath}
  An invariant periodic random set is called \emph{random periodic orbit} if it is a nonautonomous random point, i.e.~its fibers are singletons.
\end{definition}

\section{Global nonautonomous random attractors}\label{sec1}

In this section, we introduce global nonautonomous random attractors, which are invariant nonautonomous random sets that attract deterministic bounded sets. Note that in the Appendix, we develop the theory also to include the attraction of nonautonomous random sets.

\begin{definition}[Global nonautonomous random attractor]
  Let $(\theta,\Phi)$ be a nonautonomous random dynamical system on a Polish space $(X,d)$. A compact and invariant nonautonomous random set $A$ is called \emph{global nonautonomous random attractor} if for all bounded sets $C\subset X$, we have
  \begin{displaymath}
   \lim_{t\to\infty} \dist\left( \Phi(t,\tau -t,\theta_{-t}\omega)C, A(\tau,\omega)\right) = 0  \fad \tau\in \T  \text{ and almost all }  \omega\in \Omega\,,
  \end{displaymath}
  where $\dist(D_1,D_2):= \sup_{x\in D_1} \inf_{y\in D_2} d(x,y)$ is the Hausdorff semi-distance of two sets $D_1, D_2\subset X.$
\end{definition}

A sufficient condition for the existence of a global attractor is given by the existence of an \emph{absorbing set}, which is a compact nonautonomous random set $B$ such that for all bounded sets $C\subset X$, all $\tau \in \T$ and almost all $\omega\in\Omega$, there exists a time $T=T(C,\tau,\omega) >0$ such that
\begin{displaymath}
  \Phi(t,\tau -t,\theta_{-t}\omega)C\subset B (\tau,\omega)\fad t\ge T\,.
\end{displaymath}

\begin{theorem}[Existence of global nonautonomous random attractors]\label{detclosed}
  Let $(\theta,\Phi)$ be a nonautonomous random dynamical system on a Polish space $(X,d)$, and suppose that there exists an absorbing set $B.$ Then there exists a global nonautonomous random attractor $A$, given by the omega-limit set of $B$:
  \begin{displaymath}
    A(\tau, \omega) = \bigcap_{T \ge 0} \overline{\bigcup_{t\ge T} \Phi(t,\tau - t, \theta_{-t} \omega)B(\tau -t,\theta_{-t} \omega) } \fad \tau\in \T   \text{ and almost all } \omega\in \Omega\,.
  \end{displaymath}
  Note that $A$ is minimal in the sense that if there is another global nonautonomous random attractor $\tilde A$, then $A(\tau, \omega) \subset \tilde A(\tau, \omega)$ for all $\tau\in\T$ and almost all $\omega\in\Omega$.
  If $X$ is connected, then the fibers of $A$ are connected.
\end{theorem}

We prove this theorem in a more general form in the Appendix.


We now apply this result to show the existence of a nonautonomous random global attractor for a class of periodic random dynamical systems.
In particular, we consider the stochastic differential equation~\eqref{sde}, given by
\begin{displaymath}
  \rmd x=f(t,x) \rmd t+ \sigma \rmd W_t
\end{displaymath}
where $ t,x \in \R \mand  \sigma>0$, $(W_t)_{t\in\R}$ is a Wiener process and $f:\R^2\to\R$ is continuously differentiable. Let $\Phi$ denote the cocycle of the
corresponding nonautonomous random dynamical system as defined in \eqref{Phi}.
We assume the following two conditions:

\emph{1. Dissipativity condition.} There exist constants $L_1, L_2 \ge 0$ such that
\begin{equation}\label{dissipate}
  (x_1-x_2)\big(f(t,x_1)-f(t,x_2)\big) \le L_1 -L_2 \left|x_1 -x_2\right|^2\fad t\in\R \mand x_1,x_2\in\R\,.
\end{equation}

\emph{2. Integrability condition.} There exists $C_0>0$ such that
\begin{equation} \label{integrate}
  \int_{-\infty} ^{t} e^{cr}\big| f(r,u(r))\big|^2 \,\rmd r<\infty
\end{equation}
for all $t\in\R$, $0<c<C_0$ and continuous functions $u:\R\to\R$ with sub-exponential growth.

We note that these two conditions are satisfied when $f(t,x)= \alpha x- \beta x^3+ A \cos \nu t$, cf.  \eqref{sde}.

\begin {proposition}\label{srattractor}
  A nonautonomous random dynamical system generated by the SDE \eqref{sde}, where $f$ satisfies both the dissipativity and integrability condition, has a global nonautonomous random attractor.
\end {proposition}

\begin{proof}
  We prove the existence of an absorbing set in order to apply Theorem~\ref{detclosed}. The stochastic flow generated by a stochastic differential equation is, in general, not differentiable, and in order to get differentiable paths and apply techniques from deterministic calculus, we transform the stochastic differential equation \eqref{sr} into a random ordinary differential equation (see \cite{Doss_77_1,Sussmann_77_1,Sussmann_78_1,Imkeller_02_1,Kloeden_11_1}).

  Consider the one-dimensional stochastic differential equation
  \begin{equation}\label{ou}
     \rmd y= -y  \rmd t + \sigma \rmd W_t
  \end{equation}
  with the pathwise solution
  \begin{displaymath}
    \mathcal Y(t,\tau,\omega,y_\tau)=y_\tau  e^{-t}+ \sigma e^{-t} \int_{\tau}^t e^r \,\rmd W_r\,.
  \end{displaymath}
  The pullback limit of this solution is given by the Ornstein--Uhlenbeck process
  \begin{displaymath}
    O_t(\omega)=\sigma e^{-t}\int_{-\infty}^t e^r \,\rmd W_r\,,
  \end{displaymath}
  which is the unique stationary solution of \eqref{ou}. Let $Z_t:=\mathcal X_t-O_t$ for $t\in\R$, where $\mathcal X_t$ is the stochastic flow for \eqref{sde}.
  Then $t\mapsto Z_t$ is a solution of the random differential equation
  \begin{equation}\label {rode}
    \dot Z_t = f(t,Z_t+O_t)+ O_t\,.
  \end{equation}
  Let $Z(t,\tau,\omega,z)$ be the general solution of \eqref{rode} for initial time $\tau\in\R$, noise realization $\omega\in\Omega$ and initial condition $z\in \R.$ Omitting the dependence on $\tau$, $\omega$ and $z$, we obtain
  \begin{align*}
    \frac{\rmd Z^2_t}{\rmd t} & = 2 Z_t \big(f(t,\mathcal X_t)+ O_t\big)= 2 Z_t \big(f(t,\mathcal X_t)- f(t,O_t)\big)+2 Z_t \big(f(t,O_t)+ O_t\big)\\
    &\le 2 \big(L_1 -L_2 Z_t^2\big) + L_3Z_t^2 + \frac{1}{L_3}\big(f(t,O_t)+ O_t\big)^2
  \end{align*}
  for any $L_3>0$, and hence,
  \begin{equation}\label{dineq}
    \frac{\rmd Z^2_t}{\rmd t}\le -C_1 Z_t^2 + C_2 +C_3 \big(f(t,O_t)+ O_t\big)^2
  \end{equation}
  for some $C_1,C_2,C_3\ge 0.$ Note that $C_1=L_2-L_3>0$ can be chosen such that $C_1\le C_0$, with $C_0$ as in the integrability condition, and define
  \begin{displaymath}
    F(t,x):=C_2 +C_3 \big(f(t,x)+ x\big)\,.
  \end{displaymath}
  We obtain the cocycle $\Psi$ of a nonautonomous random dynamical system via
  \begin{equation}\label{Psi}
    \Psi( s, \tau,\omega )z:= Z(\tau + s ,\tau,\theta_{-\tau }\omega,z) \fad t,\tau \in \R\,, \omega \in \Omega \mand z\in \R\,,
  \end{equation}
  and the differential inequality \eqref{dineq} leads to
  \begin{equation}\label{estimate}
    \left|\Psi ( s, \tau,\omega ) z \right| ^2  \le
    \left| z\right|^2 e^{ -C_1 s} +e^{ -C_1 (s+\tau)} \int_{\tau}^{s+\tau} e^{ -C_1 r}
    F(r,O_r (\theta_{-\tau}\omega))\,\rmd r\,.
  \end{equation}
  Given a bounded set $C\subset X$, and an initial time $\tau\in\R$, a realization of noise $\omega\in\Omega$ and an initial condition $x\in C$ for the stochastic differential equation \eqref{sde}, the corresponding initial
  condition $z$  for the random differential equation \eqref{rode} is in the set
  \begin{displaymath}
    C'\fiber=C-O_\tau(\omega)= \set{ a\in \R: \text{there exists } y\in C \text{ such that } a=y-O_\tau(\omega)}\,,
  \end{displaymath}
  which defines a bounded nonautonomous random set $C'.$ Note that the inequality \eqref{estimate} implies that there exists a $T'=T'(C'(\tau,\omega)) > 0$ such that $|z|^2 e^{-C_1 s}\le 1$ for all $s>T'$ and $z\in C'(\tau,\omega).$ Due to \eqref{estimate}, we get for all $z\in C'(\tau,\omega)$
  \begin{align*}
    \left|\Psi ( s, \tau -s,\theta_{-s} \omega ) z  \right| ^2 &< 1 + e^{ -C_1 \tau} \int_{\tau -s}^\tau  e^{ -C_1  r}
    F(r,O_r (\theta_{-s}\circ \theta_{s-\tau}(\omega)))\,\rmd r\\
    & =1 + e^{ -C_1 \tau} \int_{\tau -s}^\tau e^{ -C_1  r}F(r,O_r (\theta_{-\tau}\omega)) \,\rmd r \fad s>T'(C'(\tau,\omega))\,,
  \end{align*}
  and note that the integrand does not depend on $s.$ In the limit $s\to \infty$, we obtain
  \begin{align*}
    \lim_{s\to \infty}\left|\Psi ( s, \tau -s,\theta_{-s} \omega ) z  \right| ^2 \le 1 + e^{ -C_1 \tau} \int_{-\infty}^\tau  e^{ -C_1  r}F(r,O_r (\theta_{-\tau}\omega)) \,\rmd r\,,
  \end{align*}
  where the integral is well defined because of the integrability condition. Hence, for the bounded nonautonomous random set $C'$, there exists a time $T'(C'(\tau,\omega))$ such that
  \begin{displaymath}
     \Psi ( s, \tau -s,\theta_{-s} \omega)C'\fiber \subset B\left (R(\tau,\omega)\right) \fad s >T'(C'(\tau,\omega))\,,
  \end{displaymath}
  where $B\left (R(\tau,\omega)\right)$ is the ball centered around zero with radius
  \begin{displaymath}
    R(\tau,\omega):=2 + e^{ -C_1 \tau} \int_{-\infty}^\tau  e^{ -C_1  r}F(r,O_r (\theta_{-\tau}\omega))\,\rmd r\,.
  \end{displaymath}
  Given the construction of the set $C'$, the time $T'$ depends on the deterministic bounded set $C$, $\tau$ and $\omega$, and we write $T'=T'(C,\tau,\omega).$
  Going back to the cocycle $\Phi$, for any deterministic bounded set $C$ and $s >T'(C,\tau,\omega)$, we have
  \begin{displaymath}
    \Phi ( s, \tau -s,\theta_{-s} \omega)C \subset B\left (O_\tau(\omega),R(\tau,\omega)\right)
  \end{displaymath}
  where $B\left (O_\tau(\omega), R(\tau,\omega)\right)$ is the ball  of radius $R(\tau,\omega)$ centered in $ O_\tau(\omega).$ $B\left (O_\tau(\omega), R(\tau,\omega)\right)$ is the fiber of a nonautonomous random compact set absorbing all deterministic bounded sets. This implies, by Theorem~\ref{detclosed}, that $\Phi$ has a global nonautonomous random attractor for the family of deterministic bounded sets. The attractor is a periodic, compact  and connected nonautonomous random set.
\end{proof}

\section{The global attractor is a random periodic orbit}

In this section we prove that for systems generated by the stochastic differential equation \eqref{sde}, when the deterministic forcing  is time-periodic and obeying the   dissipativity condition \eqref {dissipate} and integrability condition \eqref {integrate},  the  global nonautonomous random attractor  is a random periodic orbit.  In particular, this result can be applied to the model of   stochastic resonance given by the equation \eqref {sr}.

At the core of the argument is the existence of a correspondence between  invariant periodic measures for the nonautonomous random dynamical system and stationary periodic measures for the Markov semigroup. We will discuss these two objects in Subsections~\ref{subsec1} and \ref{subsec2} and explain the correspondence in Subsection~\ref{subsec3}.

\subsection{Invariant \na measures for the nonautonomous random dynamical system}\label{subsec1}

To define invariant measures for nonautonomous random dynamical system, we make use of the skew product flow formulation. The \emph{skew product flow} for a nonautonomous random dynamical system $(\theta,\Phi)$ is given by the mapping $\Theta:\T\times\T\times \Omega\times X\mapsto \T\times \Omega \times X$, defined by
\begin{displaymath}
  \Theta(t,\tau,\omega,x) :=(\tau+t,\theta_t \omega, \Phi(t,\tau, \omega)x)\,.
\end{displaymath}

\begin{definition}[Invariant \na measures   and invariant periodic  measures]
  Let $(\theta,\Phi)$  be a nonautonomous random dynamical system with skew product flow $\Theta.$ We say that $\mu:\T\times \cF \otimes \cB(X) \to [0,1]$
  is an \emph{invariant \na  measure} for $(\theta,\Phi)$ if
  \begin{itemize}
    \item[(i)]
    for all $\tau\in\T$, $\mu(\tau, \cdot)$ is a measure on $ \Omega \times X $ with  $\pi_{\Omega} \mu(\tau, \cdot)= \mP$,
    where $\pi_\Omega \mu(\tau, \cdot)$ denotes the marginal  on $(\Omega,\cF)$, and
    \item[(ii)]
    for all $A \in \cF \otimes \cB (X)$ and  $t,\tau\in \T$, we have
    \begin{displaymath}
      \mu(\Theta(t,\tau,A))=\mu(\tau, A)\,.
    \end{displaymath}
  \end{itemize}
   An invariant \na measure $\mu$ is called  \emph{invariant periodic measure} if there exists $T>0$ such that $$ \mu(\tau, \cdot)= \mu(\tau+T, \cdot) \quad \fa \tau\in\T\,.$$
\end{definition}

We write $\mu _\tau$ for $ \mu(\tau, \cdot).$ A measure $\mu _\tau$  on $ \Omega \times X $ with  $\pi_{\Omega} \mu _\tau= \mP$ can be uniquely  \emph{disintegrated} into a family $ \mu_{\tau,\omega}$ of probability measures on $X$ via
\begin{displaymath}
  \mu _\tau (A)= \int_\Omega  \mu_{\tau,\omega} (A_\omega)\,\rmd\mP (\omega)\,,
\end{displaymath}
where $A_\omega=\{x\in X: (x,\omega) \in A\}$ for all $A \in \cF \otimes \cB (X).$ Note that $\mu$ is an invariant \na measure if and only if
\begin{displaymath}
  \Phi(t,\tau, \omega) \mu_{\tau,\omega} = \mu_{\tau + t,\theta_t\omega} \fad t, \tau, \in \T \text{ and for almost all } \omega\in\Omega\,,
\end{displaymath}
where the measure $\mu_{\tau,\omega}$ is pushed forward, i.e.~$\Phi(t,\tau, \omega) \mu_{\tau,\omega}(C)=\mu_{\tau,\omega}(\Phi^{-1}(t,\tau, \omega) C)$ for all $C\in\cB(X).$ An invariant \na  measure $\mu$ is periodic with period $T>0$ if and only if
\begin{displaymath}
  \Phi(T,\tau, \omega) \mu_{\tau,\omega}= \mu_{\tau,\theta_T\omega} \fad \tau \in \T \text{ and for almost all } \omega\in\Omega\,.
\end{displaymath}


\begin{remark}\label{constructinvmeas}
  As for nonautonomous random sets, invariant \na  measures can be constructed easily. For instance, given a measure $\nu$ on $X$, the family of measures $\mu_{\tau,\omega}:= \Phi(\tau, 0, \theta_{-\tau}\omega)\nu$ for all $\tau \in \T$ and  $\omega\in \Omega$ is the disintegration of an invariant \na measure for the nonautonomous random dynamical system. In fact,
  \begin{align*}
    \Phi(t,\tau, \omega) \mu_{\tau,\omega} & = \Phi(t,\tau, \omega) \left(\Phi(\tau, 0, \theta_{-\tau}\omega)\nu\right)= \Phi(t+\tau,0, \theta_{-\tau}\omega) \nu=\Phi(t+\tau,0, \theta_{-(\tau+t)}\circ\theta_t\omega) \nu\\& = \mu_{t+\tau,\theta_t\omega}\,.
  \end{align*}
  Note that requiring a recurrence in time,  such as periodicity,  leads to a more meaningful concept.
\end{remark}

If there exists a global nonautonomous random attractor $A$ that is a nonautonomous random point, then $\mu_{\tau,\omega}:=\delta_{A(\tau, \omega)}$ is the disintegration of an invariant \na measure for the nonautonomous random dynamical system. 

\subsection{Stationary \na measures for nonhomogenous Markov semigroups}\label{subsec2}

We first define the concept of a stationary \na measure for nonhomogenous Markov semigroups. Suppose that $\xi(t,\tau,\omega,x)$ is the stochastic flow of the one-dimensional nonautonomous stochastic differential equation~\eqref{sde}, and let
$\rho:\T\times \cB (\R)\to [0,1]$  be a \emph{stationary \na measure} for the associated non-homogeneous Markov semigroup, i.e.
\begin{displaymath}
   \rho_{\tau +t} (B)=\int_X Q(t,\tau, x, B)\,\rmd\rho_\tau (x) \fad B\in \cB (\R) \text{ and } t,\tau \in \R\,,
\end{displaymath}
where $\rho_tau $ is $\rho(\tau, \cdot)$ and  $Q(t,\tau, x,  B)$ describe the transition probabilities for the semigroup:
\begin{equation}\label{transition}
  Q(t,\tau,x,  B):= \mP\{\omega \in \Omega: \xi(t,\tau,\omega,x)\in B\} \fad  t,\tau \in \T\,,  x\in \R \text{ and }  B\in \cB (\R)\,.
\end{equation}
We say that  $\rho $ is a \emph{stationary periodic measure} if there exists a $T>0$ such that $\rho_{\tau +T} =\rho_\tau$ for all $\tau \in \T.$

\subsection{Correspondence between invariant periodic measures and  stationary periodic  measures}\label{subsec3}

In this subsection, we extend results on the correspondence between invariant  measures and stationary measures for random dynamical systems \cite{Crauel_90_1,Crauel_91_1, Crauel_94_1, Crauel_98_1} to periodic random dynamical systems generated by the stochastic differential equation~\eqref{sde}. As a consequence, we establish conditions for the global nonautonomous random attractor to be a nonautonomous random point. The proofs are based on results for autonomous random dynamical systems.

Recall the definitions of past-time and future-time $\sigma$-algebras and of Markov measures given in \cite{Crauel_98_1}.

\begin{definition}[Past and future time $\sigma$-algebras] \label{pastfuture}
  Let $(\theta ,\phi)$ be an (autonomous) random dynamical system on the phase space $X$, with a time set $\T=\R$ or $\Z$ and a base space $(\Omega, \cF,\mP).$ The \emph{past time $\sigma$-algebra} for the random dynamical system is given by
  \begin{displaymath}
    \cF_{\le 0}:=\sigma \setb{\omega \mapsto \Phi(t,\theta _{-s} \omega) x: 0\le t \le s \text{ and } x \in X}\,.
  \end{displaymath}
  Similarly, define the \emph{future time $\sigma$-algebra} by
  \begin{displaymath}
    \cF_{\ge 0} := \sigma \setb{\omega \mapsto \Phi(-t,\theta _{s} \omega) x: 0\le t \le s \text{ and } x \in X}\,.
  \end{displaymath}
\end{definition}

\begin{definition}[Markov measure]\label{mm}
  Let $\mu$ be a measure on $(\Omega\times X,\cB (X) \otimes  \cF)$ such that $\pi_{\Omega} \mu= \mP$, let $\set{\mu_\omega}_{\omega\in\Omega}$ be its disintegration, and let $Pr(X)$ be the space of Borel probability measures on  $X$, equipped with the topology of weak convergence\footnote{The \emph{topology of weak convergence} is the smallest topology such that the mapping $\mu\mapsto\int_X f\,\rmd\mu$, on $Pr(X)\mapsto \R$, is continuous for every continuous and  bounded real function $f:X\to\R.$} and its Borel $\sigma$-algebra. A measure $\mu$ is called \emph{Markov measure} if for all $\tau\in T$, the mapping $\mu_\bullet:\Omega\mapsto Pr(X)$  is measurable with respect to the past time $\sigma$-algebra $\cF_{\le 0}.$
\end{definition}

Note that a Markov measure is not necessarily an invariant measure.

\begin{theorem}[Correspondence between invariant periodic measures and stationary periodic measures] \label{bijection}
  Suppose that the stochastic differential equation \eqref{sde} is $T$-periodic, and let $(\theta, \Phi)$ be the corresponding periodic random dynamical system. Define the discrete-time autonomous random dynamical system $(\tilde \theta, \tilde \Phi)$ by
  \begin{equation}\label{discretised}
    \tilde\Phi(n, \omega, x):=\Phi(nT, 0, \omega, x) \fad n\in \Z\,, \omega\in\Omega \text{ and } x\in \R
  \end{equation}
  and $\tilde \theta(\omega):=\theta_T(\omega)$ for all $\omega\in\Omega.$ Let $Q(t,\tau,x,  B)$ denote the transition probabilities as introduced in \eqref{transition}, and define the transition probabilities $\tilde Q(x,B):= Q(T,0,x,  B)$ for the discretised system~\eqref{discretised}.  Then there is a one-to-one correspondence $\tilde \mu \longleftrightarrow \tilde \rho$ between
  invariant Markov measures for the discrete-time random dynamical system~\eqref{discretised} and stationary measures for the discrete-time Markov semigroup defined by the transition probabilities $\tilde Q.$ In particular, if $\tilde \rho$ is a stationary measure for the discrete-time Markov semigroup, then the invariant measure $\tilde \mu$ for the discrete-time random dynamical system~\eqref{discretised} is given by
  \begin{displaymath}
    \lim_{n\to \infty} \tilde \Phi^{-1} (-n, \omega)\tilde \rho=\tilde \mu_\omega\,.
  \end{displaymath}
  The invariant measure $\tilde \mu$ can be uniquely continued to an  invariant periodic measure $\mu$ for the periodic random dynamical system $(\theta, \Phi)$, and similarly, the stationary measure $\tilde \rho$ can be uniquely continued to a stationary  periodic measure $\rho$ for the non-homogenous Markov semigroup associated to \eqref{sde}.
\end{theorem}

\begin{proof}
  The discrete-time autonomous  random dynamical system $\tilde \Phi$ is a white noise discrete random dynamical system (as defined in \cite[Section~3, p.~161]{Crauel_91_1}). It is proven in \cite{Crauel_91_1} that for white-noise systems, there is a one-to-one correspondence between invariant Markov measures and stationary measures for the corresponding Markov semigroup.

  More precisely, following \cite{Crauel_90_1}, we denote by $\tilde \theta^+$ the restriction of $\tilde \theta$ to the set $\N_0$ of non-negative integers and the probability space $\left( \Omega, \cF_{\ge0}, \mP|_{\cF_{\ge0}} \right)$, and we denote by $\tilde \Phi ^+$  the restriction of $\tilde \Phi$ to $\N_0$ and $\left( \Omega, \cF_{\ge0}, \mP|_{\cF_{\ge0}} \right).$

  If $\tilde \mu$ is an invariant Markov measure for $\tilde \Phi$, then its restriction $\tilde \mu ^+$ to $\cF_{\ge0}\otimes \cB(X)$ is invariant for $\tilde \Phi^+$, and thus, $\tilde \mu ^+$ is the product measure $\mP|_{\cF_{\ge0} }\otimes \tilde \rho$, where $\tilde \rho$  is the stationary measure for the discrete-time Markov semigroup \cite{Crauel_90_1}. Conversely, if $\tilde \rho$ is stationary for the discrete-time Markov  semigroup, then the limit
  \begin{displaymath}
     \lim_{n\to \infty} \tilde \Phi^{-1} (-n, \omega)\tilde \rho=\tilde \mu_\omega
  \end{displaymath}
  defines the disintegration of an invariant Markov measure $\tilde \mu$ for the discrete-time random dynamical system $\tilde \Phi.$

  Given the invariant measure $\tilde \mu$ for $\tilde \Phi$, we construct an invariant periodic measure for the continuous-time periodic random dynamical system $\Phi$ by pushing-forward $\tilde\mu.$ More precisely, if $\set{\tilde \mu_\omega}_{\omega\in\Omega}$ denotes the disintegration of $\tilde \mu$, then the family $\mu_{\tau,\omega}:= \Phi(\tau, 0, \theta_{-\tau}\omega)\tilde \mu_{\theta_{-\tau} \omega}$ defines an invariant periodic measure for $\Phi.$ On the other side, given the stationary measure $\tilde \rho$ for the discrete-time Markov semigroup, 
   $ \rho_{\tau } (B):=\int_X Q(\tau,0, x, B)\,\rmd \tilde \rho (x),\ \fa \tau \in \R \text{ and } B\in \cB (\R)$,
  defines  a  stationary  periodic measure for the Markov semigroup associated to the stochastic differential equation~\eqref{sde}.\footnote{Stationarity follows from the Chapman--Kolmogorov equation (see e.g. \cite[Chapter~2]{Arnold_74_1}). More precisely, for all $\tau,t \in \R \text{ and } B\in \cB (\R)$, we have $\rho_{\tau +t} (B)=\int_X Q(\tau+t,0, x, B)\,\rmd\tilde \rho (x)=\int_X \int_X Q(t,\tau, y, B) Q(\tau,0, x, \rmd y)\,\rmd\tilde \rho (x)=\int_X Q(t,\tau, y, B)\,\rmd \rho_\tau(y)$, which proves the stationarity. The last equality follows from the fact that for all $h\in L^1(X)$, we have $\int_X \int_X h(y) Q(\tau,0, x, \rmd y)\,\rmd\tilde \rho (x)=\int_X h(y) \,\rmd \rho_\tau(y).$}
\end{proof}

As a direct consequence of Theorem~\ref{bijection}, we prove the following theorem.

\begin{theorem}\label{onepointattract}
  Suppose that the stochastic differential equation \eqref{sde} is $T$-periodic, and assume that
  \begin{itemize}
    \item[(i)] there exists a unique family of stationary  $T$-periodic measures for the non-homogeneous Markov semigroup, and
    \item[(ii)] there exists a periodic global \na random attractor for the  \na random dynamical system  $\Phi$ generated  by \eqref{sde}.
 \end{itemize}
  Then $A$ is a random periodic orbit for $\Phi$ .
\end{theorem}

\begin{proof}
  Since each fiber $A(\tau, \omega)$ is a compact set, its maximum and minimum, denoted by $a_+ (\tau, \omega)$ and $a_- (\tau, \omega)$, are random periodic orbits,  due to the order-preserving property of the one-dimensional system $\Phi$. The Dirac measures $\delta_{a_- (\tau, \omega)}$ and  $\delta_{a_+ (\tau, \omega)}$ define two distinct invariant \na measures for $\Phi$:  their  restrictions to the discrete-time random dynamical system defined by \eqref {discretised} are invariant Markov measures \cite {Crauel_94_1}.
  By Theorem~\ref{bijection}, each one of the Dirac measures $a_\pm (\tau, \omega)$ corresponds to a  stationary periodic measure for the Markov semigroup, which is unique by assumption.
 Then $a_- (\tau, \omega)=a_+ (\tau, \omega)$ for all $\tau\in\R$ and for almost all $\omega \in\Omega$ and each fiber $A(\tau, \omega)$ of the attractor is a singleton,  which concludes the proof.
\end{proof}

We proved in Proposition~\ref{srattractor} that the nonautonomous random dynamical system generated by the stochastic differential equation~\eqref{sde}, with  dissipativity   and integrability conditions on the forcing, has a unique nonautonomous global random attractor. We conclude this section by proving in the periodic case that the attractor is trivial.

\begin{proposition}
  Suppose that the stochastic differential equation \eqref{sde} is $T$-periodic and the function $f$ satisfies the dissipativity condition \eqref{dissipate} and integrability condition \eqref{integrate}. Then \eqref{sde} has a uniquely determined global periodic random attractor which is a random periodic orbit.
\end{proposition}

\begin{proof}
  The nonautonomous random dynamical system fulfills the hypotheses of Theorem~\ref{onepointattract}. In fact, by Proposition~\ref{srattractor},  there exists a periodic  global \na random attractor. Given the dissipativity condition \eqref{dissipate}, we can apply the results in \cite{Veretennikov_88_1,Veretennikov_97_1} to obtain existence and uniqueness of the stationary periodic measure for the associated Markov semigroup (see \cite[Remark in Section 4]{Veretennikov_88_1} and \cite[Lemma 8]{Veretennikov_97_1}).
\end{proof}

\appendix
\section{Nonautonomous random attractors}

\setcounter{theorem}{0}
\renewcommand{\thetheorem}{A.\arabic{theorem}}

We provide a sufficient condition for the existence of a nonautonomous random attractor that attracts a family of nonautonomous random sets. This extends results obtained in \cite{Crauel_94_1, Flandoli_96_1} for random dynamical system to the case of nonautonomous random dynamical systems, and similar results for nonautonomous random dynamical systems have been obtained in \cite{Caraballo_03_3,Crauel_11_1}.

Throughout the appendix, let $(\theta:\T\times \Omega \to \Omega,\Phi:\T \times \T \times \Omega \times X \to X)$ be a nonautonomous random dynamical system on a Polish space $(X,d)$. The sufficient condition for the existence of a nonautonomous random attractor is based on so-called absorbing sets.

\begin{definition}[Absorbing set]\label{absorbing}
  A nonautonomous random set $B\subset \T \times \Omega \times X$ is called \emph{absorbing} for a nonautonomous random set $M\subset \T \times \Omega \times X$ if for all $\tau \in \T$ and for almost all $\omega\in\Omega$, there exists a time $T=T(M,\tau,\omega) >0$ such that
  \begin{displaymath}
    \Phi(t,\tau -t,\theta_{-t}\omega)M(\tau -t,\theta_{-t}\omega) \subset B (\tau,\omega) \fad t\ge T(M,\tau,\omega)\,.
  \end{displaymath}
\end{definition}

\begin{definition}[Attracting set]\label{attracting}
  An invariant nonautonomous random set $A\subset \T \times \Omega \times X$ is called \emph{attracting} for a nonautonomous random set $M\subset \T \times \Omega \times X$ if
  \begin{displaymath}
    \lim_{t\to\infty} \dist\left( \Phi(t,\tau -t,\theta_{-t}\omega)M(\tau -t,\theta_{-t}\omega), A(\tau,\omega)\right) = 0  \fad \tau\in \T \mand \text{almost all } \omega\in \Omega\,.
  \end{displaymath}
\end{definition}

We define now omega-limit sets and characterise their properties.

\begin{definition}[Omega-limit set]\label{omegalim}
  Given a nonautonomous random set $M\subset \T \times \Omega \times X$, we define
  \begin{displaymath}
    \Omega_M (\tau,\omega): = \bigcap_{T \ge 0} \overline{\bigcup_{t\ge T} \Phi(t,\tau - t, \theta_{-t} \omega)M(\tau -t,\theta_{-t} \omega) } \fad \tau\in \T \mand \omega\in \Omega\,.
  \end{displaymath}
  The set $\Omega_M:=\setb{(\tau, \omega, x)\in \T\times\Omega\times X: x\in \Omega_M (\tau,\omega) }$ is called the \emph{omega-limit set} of $M$.
\end{definition}

\begin{lemma}\label{omegainvariance}
  Let $M\subset \T \times \Omega \times X$ be a nonautonomous random set. Then the omega-limit set $\Omega_M$ is a nonautonomous random set with fibers $\Omega_M(\tau,\omega)$ as defined in Definition~\ref{omegalim}. Furthermore,
  $\Omega_M$ is \emph{forward invariant}, i.e.~we have
  \begin{displaymath}
    \Phi(t,\tau, \omega) \Omega_M (\tau,\omega) \subset \Omega_M (\tau + t,\theta_t\omega) \fad t\ge0\,,\, \tau \in \T \mand \text{almost all } \omega\in\Omega\,.
  \end{displaymath}
\end{lemma}

\begin{proof}
  Measurability of $\Omega_M$  in $\T\times\Omega\times X$ follows from the fact that $\Phi(t,\tau - t, \theta_{-t} \omega)M(\tau -t,\theta_{-t} \omega) $ is a measurable subset of $X$, and  that a  countable union of such sets is measurable. The   continuity of $\Phi$ implies the measurability of  $\Omega_M.$ To prove forward invariance, first note that
  \begin{displaymath}
    \Omega_M (\tau,\omega) = \setB{y\in X: \exists\,  t_n \to \infty, x_n \in  M(\tau -t_n,\theta_{-t_n} \omega) \text{ with } y= \lim _{n\to \infty}\Phi(t_n,\tau - t_n, \theta_{-t_n} \omega) x_n}\,.
  \end{displaymath}
  Let $y \in \Omega_M (\tau,\omega)$ and $z = \Phi(t,\tau, \omega) y$, and consider the sequences $\set{t_n}_{n\in\N}, \set{x_n}_{n\in\N}$, where $t_n \to \infty,\  x_n \in  M(\tau -t_n,\theta_{-t_n} \omega)$ such that $y= \lim _{n\to \infty}\Phi(t_n,\tau - t_n, \theta_{-t_n} \omega) x_n$. To prove that $z\in \Omega_M (\tau + t,\theta_t\omega)$, it is sufficient to find two sequences  $s_n \to \infty$ and  $z_n \in  M\opintb{\tau +t -s_n,(\theta_{-s_n}\circ \theta_{t})\omega}$ such that $z= \lim _{n\to \infty}\Phi(s_n,\tau  +t- s_n, \theta_{t-s_n} \omega) z_n$. Define $s_n:=t+t_n$. Then by continuity, we have
  \begin{align*}
  z & = \Phi(t,\tau, \omega) y= \lim _{n\to \infty} \Phi(t,\tau, \omega)\circ \Phi(t_n,\tau - t_n, \theta_{-t_n} \omega) x_n   =  \lim _{n\to \infty}  \Phi(t+t_n,\tau - t_n, \theta_{-t_n} \omega) x_n = \\
   & =\lim _{n\to \infty}\Phi(s_n,\tau  +t- s_n, \theta_{t-s_n} \omega)x_n\,.
  \end{align*}
  Since $M(\tau -t_n,\theta_{-t_n} \omega)= M(\tau +t -s_n,\theta_{t-s_n} \omega)$, we have $x_n \in  M(\tau +t -s_n,\theta_{-s_n}\circ \theta_{t}\omega)$, which completes the proof.
\end{proof}

\begin{lemma}\label{lemmacompact}
  Let $M\subset \T \times \Omega \times X$ be a nonautonomous random set and $K\subset \T \times \Omega \times X$ be a compact nonautonomous random set that is absorbing for $M$.
  Then for all $\tau \in \T$ and for almost all $\omega\in\Omega$, we have
  \begin{itemize}
    \item[(i)] $\Omega_M (\tau,\omega) \not=\emptyset$,
    \item[(ii)] $\Omega_M (\tau,\omega) \subset K(\tau,\omega)$, and $\Omega_M$ is a compact nonautonomous random set,
    \item[(iii)] $\Omega_M (\tau,\omega) \subset \Omega_K (\tau,\omega)$, and $\Omega_M$ is invariant, and  attracting for $M$. According to Definition~\ref{attracting}, this means that $\Omega_K$ attracts $M.$
  \end{itemize}
\end{lemma}

\begin{proof}
  Let $\set{t_n}_{n\in\N}, \set{x_n}_{n\in\N}$ be sequences in $\T$ and $X$ with $t_n\to \infty$ and $x_n \in  M(\tau -t_n,\theta_{-t_n} \omega)$. By the definition of absorbing set, for $n$ big enough, $y_n =\Phi(t_n,\tau - t_n, \theta_{-t_n} \omega) x_n \in K (\tau, \omega).$ $K (\tau, \omega)$ is compact, and thus, a subsequence of $\set{y_n}_{n\in\N}$ converges, which implies that
  $\Omega_M (\tau,\omega) \ne \emptyset$ and  $\Omega_M (\tau,\omega) \subset K(\tau,\omega).$

  We show now that
  \begin{displaymath}
     \Omega_M (\tau + t,\theta_t\omega) \subset \Phi(t,\tau, \omega)  \Omega_M (\tau,\omega)  \fad t, \tau \in \T \mand \omega\in\Omega\,,
  \end{displaymath}
  which, together with Lemma~\ref{omegainvariance}, proves that $\Omega_M$ is an invariant nonautonomous random set. By definition of omega-limit sets, if $z \in \Omega_M  (\tau + t,\theta_t\omega) $, then there exist two sequences $\set{t_n}_{n\in\N}, \set{x_n}_{n\in\N}$ in $\T$ and $X$ such that $t_n \to \infty$, and there exists $z_n \in  M(\tau+t -t_n,\theta_{t-t_n} \omega)$ and $z= \lim _{n\to \infty}\Phi(t_n,\tau+t -t_n,\theta_{t-t_n} \omega) z_n$.

  Define $s_n:=t_n -t.$ Then $z_n \in  M(\tau -s_n,\theta_{-s_n}\omega)$ and
  \begin{align*}
    z&= \lim _{n\to \infty}\Phi(t_n,\tau+t -t_n,\theta_{t-t_n} \omega) z_n  = \lim _{n\to \infty}  \Phi(t+s_n,\tau - s_n, \theta_{-s_n} \omega) z_n  \\
    &=\Phi(t,\tau, \omega) \lim _{n\to \infty}\Phi(s_n,\tau - s_n, \theta_{-s_n} \omega)z_n\,.
  \end{align*}
  The compactness of $K$ implies the existence of  $y=\lim _{n\to \infty}\Phi(s_n,\tau - s_n, \theta_{-s_n} \omega)z_n $, and by definition, we have $y \in  \Omega_M  (\tau,\omega)$, which proves $\Omega_M (\tau + t,\theta_t\omega) \subset \Phi(t,\tau, \omega)  \Omega_M (\tau,\omega)$.

  We now prove that $\Omega_M$ attracts $M$. By contradiction, assume that there exist $\delta >0$, a sequence $\{t_n\}_{n\in\N}$ such that $t_n\in \T$ and $t_n \to \infty$, and a sequence $\{z_n\}_{n\in\N}$ such that $z_n \in  M(\tau -t_n,\theta_{-t_n}\omega)$ and
  \begin{displaymath}
    \dist\left( \Phi(t_n,\tau -t_n,\theta_{-t_n}\omega)z_n, \Omega_M(\tau,\omega)\right) \ge \delta \fad n\in\N\,.
  \end{displaymath}

  For $n$ big enough, we have $\Phi(t_n,\tau -t_n,\theta_{-t_n}\omega)z_n \in K(\tau, \omega)$ and the limit $z=\lim _{n\to \infty}\Phi(t_n,\tau - t_n, \theta_{-t_n} \omega)z_n $ exists, at least for a suitable subsequence. By definition, $z\in \Omega_M(\tau, \omega)$, which leads to a contradiction.

  Finally, we prove that $\Omega_M(\tau, \omega) \subset \Omega_K(\tau, \omega)$. Note first that by definition,
  \begin{displaymath}
    \Omega_K(\tau, \omega)=\bigcap_{T \ge 0} \overline{\bigcup_{t\ge T} \Phi(t,\tau - t, \theta_{-t} \omega)K(\tau -t,\theta_{-t} \omega) }\,.
  \end{displaymath}
  Each $y \in \Omega_M(\tau, \omega)$ is the limit for $n\to\infty$ of $\Phi(t_n,\tau -t_n,\theta_{-t_n}\omega)x_n $, where $\set{t_n}_{n\in\N}, \set{x_n}_{n\in\N}$ are two sequences in $\T$ and $X$ such that $t_n \to \infty$ and $x_n \in  M(\tau -t_n,\theta_{-t_n} \omega)$. Denote by $T(M, \tau, \omega)$ the absorption time defined in Definition~\ref{absorbing}. Then for each $\tilde T \ge 0$, choose a sequence $t_n \ge \tilde T + T\opintb{M,\tau-\tilde T,\theta_{-\tilde T} \omega}$ for all $n\in\N$. For all $n\in\N$ and with $s_n := t_n -\tilde T \ge T\opintb{M,\tau-\tilde T,\theta_{-\tilde T} \omega}$, we have
  \begin{displaymath}
    \Phi(s_n,\tau -s_n,\theta_{-s_n}\omega)x_n \in K\opintb{\tau -\tilde T, \theta_{-\tilde T} \omega}\,.
  \end{displaymath}
  Then
  \begin{displaymath}
    \Phi(t_n,\tau -t_n,\theta_{-t_n}\omega)x_n \in \bigcup_{t\ge \tilde T} \Phi(t,\tau - t, \theta_{-t} \omega)K(\tau -t,\theta_{-t} \omega)\,.
  \end{displaymath}
  In fact, since
  \begin{align*}
    \Phi(t_n,\tau -t_n,\theta_{-t_n}\omega)x_n & = \Phi\opintb{\tilde T+s_n,\tau -t_n,\theta_{-t_n}\omega}x_n\\
    & = \Phi\opintb{\tilde T,\tau -\tilde T,\theta_{-\tilde T}\omega}\Phi(s_n,\tau -t_n,\theta_{-t_n}\omega)x_n\,,
  \end{align*}
  we have $\Phi(t_n,\tau -t_n,\theta_{-t_n}\omega)x_n \in \Phi(\tilde T,\tau -\tilde T,\theta_{-\tilde T}\omega) K(\tau -\tilde T, \theta_{-\tilde T} \omega)$. Then
  \begin{displaymath}
    \lim _{n\to \infty}\Phi(t_n,\tau -t_n,\theta_{-t_n}\omega)x_n \in \overline{\bigcup_{t\ge \tilde T} \Phi(t,\tau - t, \theta_{-t} \omega)K(\tau -t,\theta_{-t} \omega) }\,.
  \end{displaymath}
  Since $\tilde T\ge 0$ was chosen arbitrarily, we obtain
  \begin{displaymath}
    \Omega_M(\tau, \omega) \subset \bigcap_{\tilde T \ge 0} \overline{\bigcup_{t\ge \tilde T} \Phi(t,\tau - t, \theta_{-t} \omega)K(\tau -t,\theta_{-t} \omega) }  =\Omega_K(\tau, \omega)\,,
  \end{displaymath}
  which finishes the proof of this lemma.
\end{proof}

We now define global nonautonomous random attractors with respect to a family of nonautonomous random sets $\cH$ and prove a sufficient condition for its existence.

\begin{definition}[$\cH$-attractors] \label{globalattr}
  Let $\cH$ be a family of nonautonomous random sets. A invariant nonautonomous random set $A \in \cH$  is called a \emph{$\cH$-attractor} if $A$ is attracting for every $M \in \cH$.
\end{definition}

\begin{definition}[Inclusion-closed families]
  We say that a family $\cH$ of nonautonomous random sets is  \emph{inclusion-closed}  if
  \begin{itemize}
    \item[(i)] for all $M \in \cH$, the set $M(\tau, \omega)$ is non-empty for all $\tau\in \T$ and for almost all $\omega\in \Omega$,
    \item[(ii)] for all $M \in  \cH$, and for all nonautonomous random sets $\tilde M$ with
      \begin{displaymath}
        \emptyset \ne \tilde M (\tau, \omega) \subset M (\tau, \omega) \fad \tau\in \T \mand \text{almost all }\, \omega\in \Omega\,,
      \end{displaymath}
      we have $\tilde M \in  \cH$.
  \end{itemize}
\end{definition}

\begin{theorem}\label{existencecond}
  Let $\cH$ be an inclusion-closed family of random sets, and let $K \in \cH$ be a compact random set absorbing every $M \in \cH$. Then $\Omega _K$ is the unique $\cH$-attractor.
\end{theorem}

\begin{proof}
  Using Lemma~\ref{lemmacompact}, the set $\Omega_K$ is nonempty, invariant, compact, and attracts all $M \in \cH$. Since $K$ absorbs itself, we have $\Omega_K \subset K\in \cH$, and hence, $\Omega_K \in \cH$.

  To prove the uniqueness, let assume that there exist two distinct $\cH$-attractors $A,B\in \cH$. Invariance implies that
  \begin{displaymath}
    \dist(B(\tau, \omega), A(\tau,\omega)) = \dist( \Phi(t,\tau -t,\theta_{-t}\omega)B(\tau -t,\theta_{-t}\omega), A(\tau,\omega))
  \end{displaymath}
  for all $t\in\T$ and $(\tau, \omega)\in \T\times\Omega$. By definition of an attracting set, we have
  \begin{displaymath}
    \dist(\Phi(t,\tau -t,\theta_{-t}\omega)B(\tau -t,\theta_{-t}\omega), A(\tau,\omega)) \to 0 \quad \text{as } t\to\infty\,,
  \end{displaymath}
  and hence $\dist(B(\tau, \omega), A(\tau,\omega)) =0$ which implies that $B\subseteq A$. By the same argument, we obtain $\dist(A(\tau,\omega),B(\tau, \omega)) =0$. Hence $A=B$.

  This concludes the proof.
\end{proof}

We now show that Theorem~\ref{detclosed} follows directly from Theorem~\ref{existencecond}.

\begin{proof}[Proof of Theorem~\ref{detclosed}]
  Consider the omega-limit set $\Omega_B$ of the absorbing set $B$, and let $\cH$ be the family of all nonautonomous random sets which are attracted by $\Omega_B$. Then clearly, $\cH$ is inclusion-closed and contains all sets of the form $\T\times \Omega\times C$, where $C \subset X$ is bounded. Theorem~\ref{existencecond} then implies the existence of a unique $\cH$-attractor $A$. It is clear that this attractor is a global nonautonomous random attractor, since $\cH$ contains all sets of the form $\T\times \Omega\times C$, where $C \subset X$ is bounded. Minimality of the attractor can be shown with standard techniques, see, for instance, \cite[Theorem~2.12, p.~28]{Carvalho_13_1}.
\end{proof}

\bigskip\bigskip



\providecommand{\bysame}{\leavevmode\hbox to3em{\hrulefill}\thinspace}
\providecommand{\MR}{\relax\ifhmode\unskip\space\fi MR }
\providecommand{\MRhref}[2]{%
  \href{http://www.ams.org/mathscinet-getitem?mr=#1}{#2}
}
\providecommand{\href}[2]{#2}

\end{document}